\documentclass[12pt]{article}
\usepackage{graphicx}
\usepackage[ruled,vlined]{algorithm2e}
\usepackage{appendix}
\usepackage{amsmath}
\usepackage{amssymb}
\usepackage{amsthm}
\usepackage{multirow}
\usepackage{color}
\usepackage{longtable}
\usepackage{array}
\usepackage{url}
\usepackage{booktabs}
\usepackage{float}
\usepackage{mathtools}
\usepackage{tikz}
\usepackage{caption}
\usepackage{diagbox}
\usepackage{enumerate}
\usepackage{mathrsfs}

\newtheorem{theorem}{Theorem}[section]
\newtheorem{definition}[theorem]{Definition}
\newtheorem{lemma}[theorem]{Lemma}

\newtheorem{corollary}[theorem]{Corollary}

\newtheorem{problem}[theorem]{Problem}
\newtheorem{conjecture}[theorem]{Conjecture}

\theoremstyle{definition}

\title{Color isomorphic even cycles and a related Ramsey problem}
\author{Gennian Ge$^{\text{a,}}$\thanks{e-mail: gnge@zju.edu.cn. Research supported by the National Natural Science Foundation of China under Grant No. 11971325, National Key Research and Development Program of China under Grant No. 2018YFA0704703, and Beijing Scholars Program.},~ Yifan Jing$^{\text{b,}}$\thanks{e-mail: yifanjing17@gmail.com. Research partially supported by a Wind Information Scholarship.},~Zixiang Xu$^{\text{a,}}$\thanks{e-mail: zxxu8023@qq.com.}~and Tao Zhang$^{\text{c,}}$\thanks{e-mail: zhant220@163.com. Research supported by the National Natural Science Foundation of China under Grant No. 11801109.} \\
\footnotesize $^{\text{a}}$ School of Mathematics Sciences, Capital Normal University, Beijing 100048, China.\\
\footnotesize $^{\text{b}}$ Department of Mathematics, University of Illinois at Urbana Champaign, Urbana, Illinois, 61801, USA. \\
\footnotesize $^{\text{c}}$ School of Mathematics and Information Science, Guangzhou University, Guangzhou 510006, China.\\}

\begin{document}

\date{}

\maketitle

\begin{abstract}
 In this paper, we first study a new extremal problem recently posed by Conlon and Tyomkyn~(arXiv: 2002.00921). Given a graph $H$ and an integer $k\geqslant 2$, let $f_{k}(n,H)$ be the smallest number of colors $c$ such that there exists a proper edge-coloring of the complete graph $K_{n}$ with $c$ colors containing no $k$ vertex-disjoint color-isomorphic copies of $H$. Using algebraic properties of polynomials over finite fields, we give an explicit proper edge-coloring of $K_{n}$ and show that $f_{k}(n, C_{4})=\Theta(n)$ when $k\geqslant 3$ and $n\rightarrow\infty$. The methods we used in the edge-coloring may be of some independent interest. We also consider a related generalized Ramsey problem. For given graphs $G$ and $H,$ let $r(G,H,q)$ be the minimum number of edge-colors (not necessarily proper) of $G$, such that the edges of every copy of $H\subseteq G$ together receive at least $q$ distinct colors. Establishing the relation to the Tur\'{a}n number of specified bipartite graphs, we obtain some general lower bounds for $r(K_{n,n},K_{s,t},q)$ with a broad range of $q$.
\end{abstract}

\medskip

\noindent {{\it Key words and phrases\/}: Generalized Ramsey number, edge-coloring, algebraic construction.}

\smallskip

\noindent {{\it AMS subject classifications\/}: 05C15, 05C35, 11T06.}

\section{Introduction}
  The problems of finding rainbow structures in proper edge-colorings of complete graphs have been widely studied in recent years. For example, in a recent breakthrough by Montgomery, Pokrovskiy and Sudakov~\cite{ProofRingel2020}, they confirmed the famous Ringel's conjecture, one of whose statements involves finding a rainbow copy of any tree with $n$ edges in a particular proper edge-coloring of $K_{2n+1}.$ There have also been many papers written on finding large or spanning structures in proper edge-colorings, see, e.g. \cite{Stefan2019, PLMS2019, JEMS2020, JCTB2018}. Very recently, Conlon and Tyomkyn \cite{Conlon2020} studied a new Ramsey type problem, which aims to find two or more vertex-disjoint color-isomorphic copies of some given graph in proper edge-colorings of complete graphs.

For $k,n\geqslant 2$ and a fixed graph $H$, define $f_{k}(n,H)$ to be the smallest integer $c$ such that there exists a proper edge-coloring of $K_{n}$ with $c$ colors containing no $k$ vertex-disjoint color-isomorphic repeats of $H$. One may ask the following natural question.

\begin{problem}\label{problem:MAIN}
Given $k\geqslant 2$ and a fixed graph $H,$ determine the order of growth of $f_{k}(n,H)$ as $n\rightarrow\infty.$
\end{problem}

In \cite{Conlon2020}, Conlon and Tyomkyn studied this problem systematically. They first made many useful observations on the properties of $f_k(n,H)$. For instance, $f_{k}(n,H)$ is monotone increasing in $n$, but decreasing in $k$. Moreover, $f_{k}(n,H)$ is monotone decreasing in $H$ with respect to taking subgraph, i.e, $f_{k}(n,H)\leqslant f_{k}(n,H')$ when $H'$ is a subgraph of $H.$ Also, since every proper coloring of $K_n$ uses at least $n-1$ colors, then $n-1\leqslant f_{k}(n,H)\leqslant \binom{n}{2}.$ Using the Lov\'{a}sz Local Lemma and Bukh's random algebraic method~\cite{Bukh2015}, they proved the following results.
\begin{theorem}[\cite{Conlon2020}]\label{thm:SomeknownResults} The followings hold.\medskip

 \noindent\emph{(i)} For any graph $H$ with $v$ vertices and $e$ edges,
  \begin{equation*}
    f_{k}(n,H)=O(\max\{n,n^{\frac{kv-2}{(k-1)e}}\}).
    \end{equation*}

    \noindent\emph{(ii)} For every graph $H$ containing a cycle, there exists $k=k(H)$ such that
\begin{equation*}
  f_{k}(n,H)=\Theta(n).
\end{equation*}
\end{theorem}
Conlon and Tyomkyn also suggested to study $f_k(n,H)$ when $H$ is an even cycle. Theorem~\ref{thm:SomeknownResults} implies $f_{k}(n,C_{4})=O(n^{\frac{2k-1}{2k-2}})$ (using the Lov\'{a}sz Local Lemma), and there is an integer $k$ such that $f_{k}(n,C_{4})=\Theta(n)$ (using the random algebraic method). The constant $k$ obtained by the random algebraic method is likely very large due to the Lang-Weil bound \cite{LangWeil1954}, and they asked whether $f_{2}(n,C_{4})=\Theta(n).$

Our first result in this paper studies $f_k(n,C_4)$. We try to estimate the smallest integer $k$ such that $f_{k}(n,C_{4})=\Theta(n)$, and we give the following result via an algebraic construction.

\begin{theorem}\label{thm:f12C4}
$f_{3}(n,C_{4})=\Theta(n).$
\end{theorem}

This result improves the best known upper bound $O(n^{\frac{5}{4}})$ obtained by Theorem~\ref{thm:SomeknownResults} (i), and it perhaps gives some evidence that $f_{2}(n,C_{4})$ is also of order $\Theta(n)$.

As the authors mentioned in \cite{Conlon2020}, the problem of studying $f_k(n,H)$ was motivated by a generalized Ramsey problem raised by Krueger~\cite{GenRamsey2020}. Our next result in this paper studies this generalized Ramsey problem.
The classical graph Ramsey problem asks for the minimum number $n$ such that every $k$-coloring of the edges of $K_{n}$ forces a monochromatic copy of $K_{p}$. By fixing $n,$ the inverse problem asks for the minimum $k$ such that there exists an edge-coloring of $K_{n}$ with $k$ colors, and each copy of $K_{p}$ receives at least $q$ colors. For general graphs $G$ and $H,$ the generalized Ramsey number $r(G,H,q)$ denotes the minimum number of edge-colors of $G$, such that the edges of every copy of $H\subseteq G$ together receive at least $q$ distinct colors. This function was first studied by Elekes, Erd\H{o}s and F\"{u}redi (see Section $9$ of \cite{Erdos1981}). Later, Erd\H{o}s and Gy\'{a}rf\'{a}s~\cite{Combinatorica1997} systematically studied the function $r(K_{n},K_{p},q)$ and showed many improved results. After that, a number of wonderful papers \cite{JCTB2000,  PLMS2015, SIDMA2020, MonoK42008, ColorEnergy2019, EJC2001}
studied this problem, and obtained lots of interesting results of this fashion.
In recent years, some questions about distinct distances and difference sets with similar flavors have also been studied in \cite{SIDMA2020, DistinctDistance2018, ColorEnergy2019}.

In this paper, we are interested in the bipartite version of generalized Ramsey number $r(K_{n,n},K_{s,t},q),$ which has been studied in \cite{JCTB2000, JCTB1975, ARS2003}. In particular,  Axenovich, F\"{u}redi, and Mubayi~\cite{JCTB2000} obtained a series of improved results via many different methods such as Lov\'{a}sz Local Lemma and algebraic methods. We list some results of \cite{JCTB2000} in Tables~\ref{table:K2233} and \ref{table:Kss}.

In our studies of $r(K_{n,n},K_{s,t},q)$, we will always assume that $s,t$ and $q$ are fixed integers and $n\rightarrow \infty.$ Axenovich, F\"{u}redi, and Mubayi~\cite{JCTB2000} determined the linear and quadratic thresholds of the function $r(K_{n,n},K_{s,s},q)$. More precisely, they determined the smallest integers $q_1(s)$ and $q_2(s)$, where $q_1(s)=s^{2}-2s+3$ and $q_2(s)=s^{2}-s+2$, such that $r(K_{n,n},K_{s,s},q_1(s))=\Theta(n)$ and $r(K_{n,n},K_{s,s},q_2(s))=\Omega(n^{2})$. Up to now, nothing has been shown about $r(K_{n,n},K_{s,s},q)$ beyond the trivial lower bound $\Omega(n)$ when $s^{2}-2s+3<q<s^{2}-s+2.$

Our next results give some general lower bounds for $r(K_{n,n},K_{s,t},q)$ with a broad range of $q$. Recall that $\mathrm{ex}(n,H)$ denotes the maximum number of edges in an $H$-free graph $G$ with $n$ vertices.

\begin{theorem}\label{thm:GeneralLowerBound}
  For given integers $t\geqslant s\geqslant 4,$  we have
  \begin{equation*}
    r(K_{n,n},K_{s,t},st-e(H)+1)=\Omega\Big(\frac{n^{4}}{\textup{ex}(n^{2},H)}\Big),
  \end{equation*}
  where $H$ is a bipartite graph with bipartition $H=H_{1}\bigcup H_{2}$ such that $|H_{1}|\leqslant \lfloor\frac{s}{2}\rfloor$ and $|H_{2}|\leqslant \lfloor\frac{t}{2}\rfloor.$
\end{theorem}

When $s=t$ is even, let $H$ be the even cycle of length $s.$ Using the upper bound $\textup{ex}(n,C_{2k})=O(n^{1+\frac{1}{k}})$ by Bondy~\cite{Bondy1974}, we obtain the following corollary.

\begin{corollary}\label{cor:kss}
When $s\geqslant 4$ is an even integer, we have
  \begin{equation*}
    r(K_{n,n},K_{s,s},s^{2}-s+1)=\Omega(n^{2-\frac{4}{s}}).
  \end{equation*}
\end{corollary}
As we see in Table~\ref{table:Kss}, our lower bound is not far from the best known upper bound $r(K_{n,n},K_{s,s},s^{2}-s+1)=O(n^{2-\frac{2}{s}}),$ and Corollary~\ref{cor:kss} gives the first non-trivial lower bound when $s>4$ is an even integer and $q=s^{2}-s+1.$ Note that the choice of $H$ in Theorem~\ref{thm:GeneralLowerBound} is flexible, which leads to a broad range of $q$ in function $r(K_{n,n},K_{s,t},q).$ Let $q_1(s,t)=st-s-t+3$ and $q_2(s,t)=st-\frac{s+t}{2}+2$. The results in \cite{GenRamsey2020} imply $r(K_{n,n},K_{s,t},q_1(s,t))=\Theta(n)$ and $r(K_{n,n},K_{s,t},q_2(s,t))=\Theta(n^2)$. By choosing different graphs $H$ satisfying the conditions in Theorem~\ref{thm:GeneralLowerBound}, as corollaries, we are able to get super-linear lower bounds for $r(K_{n,n},K_{s,t},q)$ with many distinct parameters $q.$

We give two corollaries in this fashion for the asymmetric version, that is, when $t>s$. The corollaries are obtained by choosing $H$ as $1$-subdivision of complete graphs and $1$-subdivision of complete bipartite graphs, respectively. We utilise known extremal numbers of $1$-subdivision of complete graphs $\textup{ex}(n,\textup{sub}(K_{t}))=O(n^{\frac{3}{2}-\frac{1}{4t-6}})$, and $1$-subdivision of complete bipartite graphs $\textup{ex}(n,\textup{sub}(K_{s,t}))=O(n^{\frac{3}{2}-\frac{1}{2s}})$ (see \cite{ConlonJanzerLee2019, JanzerEJC2019}).

\begin{corollary}
  When $s\geqslant 6$ is an even integer and $t=2\binom{s/2}{2}$, we have
  \begin{equation*}
    r(K_{n,n},K_{s,t},st-t+1)=\Omega(n^{1+\frac{1}{s-3}}).
  \end{equation*}
\end{corollary}

\begin{corollary}
  When $s\geqslant 8$ is an even integer and $t=\frac{s^{2}}{8}$, we have
  \begin{equation*}
    r(K_{n,n},K_{s,t},st-t+1)=\Omega(n^{1+\frac{4}{s}}).
  \end{equation*}
\end{corollary}

Also, there are not many lower bounds for $r(K_{n,n},K_{s,t},q)$ that have been found when $q<q_{1}(s,t)=st-s-t+3.$ The only known case is when $q=2,$ $r(K_{n,n},K_{s,s},2)=\Omega(n^{\frac{1}{s}})$~\cite{JCTB2000}. We obtain some new sub-linear lower bounds by using Theorem~\ref{thm:GeneralLowerBound} and the famous K\"{o}vari-S\'{o}s-Tur\'{a}n bound~\cite{Kovari1954} $\textup{ex}(n,K_{m,\ell})=O(n^{2-\frac{1}{m}})$ with $m\leqslant \ell$.

 \begin{corollary}
   When $t\geqslant s\geqslant 4,$ let $m\leqslant \frac{s}{2}$ and $\ell\leqslant\frac{t}{2}$ be integers with $m\leqslant\ell,$ we have
   \begin{equation*}
     r(K_{n,n},K_{s,t},st-m\ell+1)=\Omega(n^{\frac{2}{m}}).
   \end{equation*}
   In particular, set $s=t$ and $m=\ell=\frac{s}{2},$ we have
   \begin{equation*}
     r(K_{n,n},K_{s,s},\frac{3s^{2}}{4}+1)=\Omega(n^{\frac{4}{s}}).
   \end{equation*}
 \end{corollary}

The paper is organized as follows. In Section~\ref{section:C4}, we prove Theorem~\ref{thm:f12C4} by giving an algebraic construction. In Section~\ref{section:lowbounds}, we prove Theorem~\ref{thm:GeneralLowerBound}. Finally we conclude with some remarks and further questions in Section~\ref{section:Conclusion}.\smallskip

\begin{minipage}{\textwidth}
 \begin{minipage}[t]{0.5\textwidth}
  \centering
     \makeatletter\def\@captype{table}\makeatother\caption{$r(K_{n,n},K_{s,s},q)$ with $s=2,3$}\label{table:K2233}
       \begin{tabular}{ccc}
\hline
$q$& $r(K_{n,n},K_{2,2},q)$& $r(K_{n,n},K_{3,3},q)$ \\
\hline
$2$& $(1+o(1))\sqrt{n}$&$(1+o(1))n^{\frac{1}{3}}$\\
$3$& $>\lfloor\frac{2n}{3}\rfloor$; $\leqslant n-1$ & $O(n^{\frac{4}{7}})$\\
$4$&$n^{2}$&$O(n^{\frac{2}{3}})$\\
$5$&$n^{2}$&$O(n^{\frac{4}{5}})$\\
$6$&$n^{2}$&$\Theta(n)$\\
$7$&$n^{2}$&$\Omega(n)$; $O(n^{\frac{4}{3}})$\\
$8$&$n^{2}$&$\lceil\frac{n}{2}\lceil\frac{3n}{2}\rceil\rceil$\\[0.5mm]
$9$&$n^{2}$&$n^{2}$\\
\hline
\end{tabular}
  \end{minipage}
  \begin{minipage}[t]{0.5\textwidth}\label{table:Kss}
   \centering
        \makeatletter\def\@captype{table}\makeatother\caption{$r(K_{n,n},K_{s,s},q)$ with $s\geqslant 4$}\label{table:Kss}
        \begin{tabular}{cc}
\hline
$q$& $r(K_{n,n},K_{s,s},q)$ \\
\hline
$2$& $\Omega(n^{\frac{1}{s}})$\\
$s^{2}-2s+2$& $O(n^{1-\frac{1}{2s-1}})$\\[0.9mm]
$s^{2}-2s+3$&$\Theta(n)$\\
$s^{2}-s+1$&$O(n^{2-\frac{2}{s}})$\\
$s^{2}-s+2$&$\geqslant C_{s}(n^{2}-n)$; $<(1-c_{s})n^{2}$\\
$s^{2}-\lfloor\frac{2s-1}{3}\rfloor+1$&$>n^{2}-2\lfloor\frac{s-2}{3}\rfloor(n-1)$\\
$s^{2}-\lfloor\frac{s}{2}\rfloor+1$&$n^{2}-\lfloor\frac{s}{2}\rfloor+1$\\[0.5mm]
$s^{2}$&$n^{2}$\\
\hline
\end{tabular}
   \end{minipage}
\end{minipage}

\section{Even cycle $C_{4}$}\label{section:C4}

In this section, we are going to prove Theorem~\ref{thm:f12C4}. The proof goes as follows. We first choose a field $\mathcal{K}$, and construct a map $\pi:V\to \mathcal{K}$, where $V=V(K_n)$. Then we choose a symmetric polynomial $P\in \mathcal{K}[x,y]$, and color the edge $ab$ by $P(\pi(a),\pi(b))$. We aim to show that under this construction, the edge-coloring we obtained has bounded maximum degree in each color class (thus by a standard application of Vizing's theorem, we are able to get a proper edge coloring), the image $|P(\pi(V),\pi(V))|$\footnote{Given $A,B\subseteq \mathcal{K}$, let $P(A,B)$ denote $\{P(a,b): a\in A,b\in B\}$.} is $O(n)$, and we cannot find too many color isomorphic copies of $C_4$ in this coloring.

One may choose $\mathcal{K}$ a field of characteristic $0$. Thus, by the symmetric Elekes--Ronyai theorem~\cite{JRT}, our symmetric polynomial $P(x,y)$ has the form $f(u(x)+u(y))$ or $f(u(x)u(y))$, where $f,u$ are some one variable polynomials in $\mathcal{K}[x]$. However, constant many color isomorphic copies of $C_4$ in this case would imply the set $u(\pi(V))$ has low additive or multiplicative energy, and this gives us that the image $|P(\pi(V),\pi(V))|$ is $\Theta(\pi(V)^2)$, providing $\pi(V)=O(n^{\frac{1}{2}})$. Hence the maximum degree of the color class is $\Omega(n^{\frac{1}{2}})$, we will no longer have a proper edge coloring.

We fix this problem by choosing $\mathcal{K}$ to be a finite field $\mathbb F_p$, and we choose our polynomial $P$ to be an expanding polynomial over $\mathbb R$. This means, if we view our polynomial as an element in $\mathbb R[x,y]$, and $\pi:V\to \mathbb R$, by taking an injection $\pi$, the edge coloring is proper, we will not have many color isomorphic copies of $C_4$, but the image $|P(\pi(V),\pi(V))|$ is quadratic in $n$. Now, by taking $p=O(n)$, mapping everything down to $\mathbb F_p$ will help us to decrease $|P(\pi(V),\pi(V))|$ to $O(n)$. By choosing the polynomial $P(x,y)$, the vertex map $\pi$, and the character $p$ carefully, and using the ideas from the properties of the resultants of polynomials, we manage to show that, we can still get a proper coloring, and taking this projection down to $\mathbb F_p$ will not create too many color isomorphic copies of $C_4$.

\begin{proof}[Proof of Theorem~\ref{thm:f12C4}]
Let $p\equiv5\pmod{6}$ be a sufficiently large prime, and
\[ A=\Big\{1,2,\ldots,\Big\lfloor\frac{p-1}{3}\Big\rfloor\Big\}
\]
be a subset of $\mathbb{F}_{p}$.

We remark that the purpose of choosing $p\equiv5\pmod{6}$ is to make $-3$ a non-residue modulo $p$. This follows from the law of quadratic reciprocity, which states that if $r$ and $s$ are odd primes, and we define $r^{*}$ to be $r$ if $r\equiv1\pmod{4}$ or $-r$ if
$r \equiv3\pmod{4},$ then $s$ is a quadratic residue modulo $r$ if and only if $r^{*}$ is a quadratic residue modulo
$s$. Applying this with $r=3$ and $s=p$, we see that $-3$ is a quadratic residue modulo $p$ if and only
if $p$ is a quadratic residue modulo $3$, i.e. $p\equiv1\pmod{3}$.

Now, let $G$ be a complete graph with vertex set $A$. Let $P(x,y)=x^2+xy+y^2$ be an element in $\mathbb F_p[x,y]$. Let $P^*(A,A)$ be the restricted image, that is
\[
P^*(A,A)=\{P(a,b):a,b\in A,a\neq b\}.
\]
Define the edge coloring
\[
\chi: E(G)\to P^*(A,A),
\]
such that for every $x,y\in A$ with $x\neq y$, we assign color $P(x,y)$ to the edge $xy$. Note that $x^{2}+xy+y^{2}=(x+\frac{y}{2})^{2}+\frac{3}{4}y^{2}$, and $-3$ is a non-quadratic residue modulo $p$, then $x^{2}+xy+y^{2}\ne0$. Hence $P^*(A,A)\subseteq \mathbb F^*_p$.

Next, we claim that the edge coloring $\chi$ is proper. Suppose there is $a\in\mathbb F_p^*$, such that two edges $xy$ and $xz$ are assigned to the same color $a$. That is, we have $x^{2}+xy+y^{2}=a$ and $x^{2}+xz+z^{2}=a$. Hence $(y-z)(x+y+z)=0$. By the way we construct the vertex set $A$, $x+y+z\ne0$. This implies $y=z$.  Therefore, two distinct edges $xy$ and $xz$ can not be assigned the same color, then $\chi$ is proper.

Let $a,b,c,d\in\mathbb{F}^*_{p}$, and $a\ne b,\ b\ne c,\ c\ne d,\ d\ne a$. Assume the colors $a,b,c,d$ are incident to a four-cycle $x,y,z,w\in A$. Then we have
\begin{align}
&x^{2}+xy+y^{2}=a,\label{c4eq1}\\
&y^{2}+yz+z^{2}=b,\label{c4eq2}\\
&z^{2}+zw+w^{2}=c,\label{c4eq3}\\
&w^{2}+wx+x^{2}=d.\label{c4eq4}
\end{align}

From Equations (\ref{c4eq1}) and (\ref{c4eq2}), we obtain
\begin{equation}\label{c4eq5}
    y=\frac{a-b}{x-z}-x-z.
\end{equation}
Similarly, from Equations (\ref{c4eq3}) and (\ref{c4eq4}), we obtain
\begin{equation}\label{c4eq6}
    w=\frac{c-d}{z-x}-z-x.
\end{equation}
Substituting Equations (\ref{c4eq5}) and (\ref{c4eq6}) into Equations (\ref{c4eq1}) and (\ref{c4eq4}), respectively, we have
\begin{align}
&x^2+x\Big(\frac{a-b}{x-z}-x-z\Big)+\Big(\frac{a-b}{x-z}-x-z\Big)^{2}=a,\label{c4eq7}\\
&x^2+x\Big(\frac{c-d}{z-x}-z-x\Big)+\Big(\frac{c-d}{z-x}-z-x\Big)^{2}=d.\label{c4eq8}
\end{align}
Making the change of variables $u\mapsto x+z$ and $v\mapsto x-z$ in Equations (\ref{c4eq7}) and (\ref{c4eq8}), we get $v\neq 0$, and
\begin{align}
&v^{4}+3u^{2}v^{2}-6(a-b)uv-2(a+b)v^{2}+4(a-b)^{2}=0,\label{c4eq9}\\
&v^{4}+3u^{2}v^{2}-6(c-d)uv-2(c+d)v^{2}+4(c-d)^{2}=0.\label{c4eq10}
\end{align}

Assume that $a-b=c-d$. Then from Equations (\ref{c4eq9}) and (\ref{c4eq10}), we have $a+b=c+d$, hence $a=c,b=d$. By Equations (\ref{c4eq5}) and (\ref{c4eq6}), we have $y+w+2x+2z=0$, and by the symmetry, we can also obtain $x+z+2y+2w=0$. Therefore, $x+z=y+w=0$, which contradicts the construction of $A$.

Now we have that $a-b\ne c-d$. From Equations (\ref{c4eq9}) and (\ref{c4eq10}) we can obtain
\begin{align}
uv=\frac{(a+b-c-d)v^{2}}{3(c-d-a+b)}+\frac{2}{3}(c-d+a-b).\label{c4eq11}
\end{align}
Substituting Equation (\ref{c4eq11}) into Equation (\ref{c4eq9}), and then multiplying by $3(c-d-a+b)^{2}$, we obtain
\begin{align}
k_{2}v^{4}+k_{1}v^{2}+k_{0}=0,\label{eq:quad}
\end{align}
where
\begin{align*}
&k_{2}=4(a^2-ab-2ac+ad+b^2+bc-2bd+c^2-cd+d^2),\\
&k_{1}=4(a-b-c+d)(a^2-2ac-3ad-b^2+3bc+2bd+c^2-d^2),\\
&k_{0}=4(a-b-c+d)^2(a^2-2ab-ac+ad+b^2+bc-bd+c^2-2cd+d^2).
\end{align*}

Assume first that $k_{2}=k_{1}=k_{0}=0$. Since $a-b-c+d\ne0$, we have
\begin{align}
&a^2-ab-2ac+ad+b^2+bc-2bd+c^2-cd+d^2=0,\label{c4eq12}\\
&a^2-2ab-ac+ad+b^2+bc-bd+c^2-2cd+d^2=0.\label{c4eq13}
\end{align}
From Equations (\ref{c4eq12}) and (\ref{c4eq13}), we obtain $(a-d)(b-c)=0$, which contradicts the choice of $a,b,c,d$.

Recall that $v=x-z$ and $x,z\in A$.
Thus if at least one of $k_{i}$ $(i=0,1,2)$ is not $0$, then there are at most $4$ solutions for $v$ in $(A-A)\setminus\{0\}$. Also, if the number of solutions for $v$ is at least $3$, then there exist two solutions $v_{1},v_{2}$ such that $v_{1}+v_{2}=0$. In this case, let $x_{1},z_{1},u_{1}$ ($x_{2},z_{2},u_{2}$) be the corresponding solutions with respect to $v_{1}$ ($v_{2}$, respectively). Then $v_{1}^{2}=v_{2}^{2}$, and by Equation (\ref{c4eq11}), we have $u_{1}v_{1}=u_{2}v_{2}$. Hence $u_{1}=-u_{2}$. Note that since $v_{i}=x_{i}-z_{i}$ and $u_{i}=x_{i}+z_{i}$, we have $x_{1}-z_{1}=-x_{2}+z_{2}$ and $x_{1}+z_{1}=-x_{2}-z_{2}$. Then we can get $x_{1}+x_{2}=0$, which contradicts the construction of $A$.

Hence there are at most $2$ solutions for $v$ in $(A-A)\setminus\{0\}$. For any fixed $v$, by Equation (\ref{c4eq11}), there is a unique solution for $u$. Note that $x=\frac{u+v}{2}$ and $z=\frac{u-v}{2}$,
 and by Equations (\ref{c4eq5}) and (\ref{c4eq6}), $y$ and $w$ are uniquely determined by $x$ and $z$, there are at most $2$ solutions for $(x,y,z,w)$. Therefore, the number of copies of four-cycles with edge colors $a,b,c,d$ is at most two.
\end{proof}

\section{Lower bounds for $r(K_{n,n},K_{s,t},q)$}\label{section:lowbounds}
In this section, we are going to prove Theorem~\ref{thm:GeneralLowerBound}. The ideas used in this proof are mainly inspired by the recent work of Conlon and Tyomkyn~\cite{Conlon2020}. It will often be helpful to think of $r(K_{n,n},K_{s,t},q)$ in terms of repeated colors. Let $\mathscr{C}$ be the collection of colors, and let $\chi:E(G)\to\mathscr{C}$ be an edge coloring of graph $G$. Let $H$ be a subgraph of $G$. If a color
$c\in\mathscr{C}$ appears on exactly $r_{c}$ edges in $\chi(E(H))$, then we say such color $c$ is repeated $r_{c}-1$
times in $H$. We say $H$ has $r$ \emph{repeats} if $r=\sum\limits_{c\in\chi(E(H))}(r_{c}-1)$, where every color $c\in\chi(E(H))$ is repeated $r_{c}-1$ times, and the sum is taking over all colors in $\chi(E(H))$ (hence $r_c\geqslant1$).

\begin{proof}[Proof of Theorem~\ref{thm:GeneralLowerBound}]
  Suppose that $n$ is sufficiently large, and let
  \[
  \chi:E(K_{n,n})\to \mathscr{C}
  \]
  be an edge coloring of $K_{n,n}$, where $\mathscr{C}$ is the collection of colors. Suppose $K_{n,n}$ has the vertex bipartition $A\cup B$. We label the vertices in $A$ and $B$ respectively, such that $A=\{a_{1},a_{2},\ldots,a_{n}\}$ and $B=\{b_{1},b_{2},\ldots,b_{n}\}$, with $a_{1}<a_{2}<\cdots<a_{n}$ and $b_{1}<b_{2}<\cdots<b_{n}$.

  Now we construct the auxiliary graph $F$ as follows. $F$ is a bipartite graph, with vertex set $U\cup W$, such that $U=\binom{A}{2}$ and $W=\binom{B}{2}$. Thus $|U|=|W|=\binom{n}{2}$. Moreover, we require the elements in $U$ to have the form $(a_i,a_j)$ with $a_i>a_j$, and elements in $W$ to have the form $(b_k,b_\ell)$ with $b_k>b_\ell$. For every $(a_i,a_j)\in U$ and $(b_k,b_\ell)\in W$, $(a_i,a_j)$ and $(b_k,b_\ell)$ are adjacent in $F$ if $\chi(a_ib_k)=\chi(a_jb_\ell)$ in the edge coloring of $K_{n,n}$. Given $c\in\mathscr{C}$, let $e_c$ be the number of edges of color $c$ in the image of $\chi$, we have
  \[
  e(F)=\sum_{c\in\mathscr{C}}\binom{e_{c}}{2}\geqslant \frac{(\sum_{c\in\mathscr{C}}e_{c})^{2}}{4|\mathscr{C}|}=\frac{n^{4}}{4|\mathscr{C}|}.
  \]
  Hence $|\mathscr{C}|\geqslant \frac{n^4{}}{4e(F)}$.

  Next, we are going to bound $e(F)$, and hence get a lower bound on $|\mathscr{C}|$. Let $H$ be a bipartite graph, with vertex set $H_1\cup H_2$, such that
  $|H_{1}|\leqslant \lfloor\frac{s}{2}\rfloor$ and $|H_{2}|\leqslant \lfloor\frac{t}{2}\rfloor.$ Suppose $e(F)\geqslant \mathrm{ex}(|V(F)|,H)$, then $F$ contains a copy of $H$. Observe that, by the definition of auxiliary graph $F$, every edge of the copy $H$ in $F$ will contribute exactly one repeat in the edge coloring $\chi$ of $K_{n,n}$. Thus, there are at least $e(H)$ repeats in $\chi$. Moreover, all these $e(H)$ repeats span at most $2|H_1|$ vertices in $A$, and at most $2|H_2|$ vertices in $B$. Thus by the upper bound on $|H_1|$ and $|H_2|$, we are able to find a copy of $K_{s,t}$ in $K_{n,n}$ such that $|\chi(E(K_{s,t}))|$ is at most $st-e(H)$. Therefore, if the image of $\chi$ does not contain a $K_{s,t}$ with less than $st-e(H)+1$ colors, we have $e(F)\leqslant \mathrm{ex}(|V(F)|,H)$, finishing the proof.
\end{proof}

\section{Concluding remarks}\label{section:Conclusion}
 Although the proof of Theorem~\ref{thm:f12C4} only requires an elementary computation, it is motivated by considering the resultants of polynomials. Let us first recall the definition of the resultant of polynomials over $\mathcal{K}[x]$.
\begin{definition}
Let $f(x),g(x)\in \mathcal{K}[x]$, such that $f(x)=a_{m}x^{m}+\cdots+a_{1}x+a_{0}$ and $g(x)=b_{n}x^{n}+\cdots+b_{1}x+b_{0}$. Then the \emph{resultant of $f$ and $g$} is defined by the determinant of the following $(m+n+2)\times (m+n+2)$ matrix,
\begin{align*}
\left(
            \begin{array}{cccccccc}
              a_{0} & a_{1} & \cdots & a_{m} &   &   &&   \\
               & a_{0} & \cdots  & a_{m-1} & a_{m} && & \\
                &   & \cdots  & \cdots  & \cdots & &  &\\
              &  &   &   &   & a_{0} & \cdots  & a_{m}\\
              b_{0} & b_{1} & \cdots &\cdots & b_{n} &     &\\
               & b_{0} & \cdots & \cdots  & b_{n-1} & b_{n} & & \\
                &   &  \cdots &  \cdots &\cdots  & \cdots  &   &\\
                &   &   &   &b_{0} & \cdots& \cdots  & b_{n}\\
            \end{array}
          \right),
\end{align*}
which is denoted by $R(f,g)$.
\end{definition}
The resultant of two polynomials has the following property, which is crucial.
\begin{lemma}[\cite{Fuhrmann2012}]
Let $f,g\in \mathcal{K}[x]$ be one variable polynomials. Suppose $h(x)=\gcd(f(x),\allowbreak g(x))$, where $\deg(h(x))\geqslant 1$. Then $R(f,g)=0$. In particular, if $f$ and $g$ have a common root in $\mathcal{K}$, then $R(f,g)=0$.
\end{lemma}

For the multivariable polynomials, the above lemma still holds when we project down to a one variable polynomial ring. For any $f,g\in \mathcal{K}[x_{1},\dots,x_{n}]$, let $R(f,g;x_{i})$ denote the resultant of $f$ and $g$ with respect to the variable $x_{i}$.

In the proof of Theorem~\ref{thm:f12C4}, Equations (\ref{c4eq1}), (\ref{c4eq2}), (\ref{c4eq3}), and (\ref{c4eq4}) actually give us four polynomials in $\mathbb F_p[x,y,z,w]$
\begin{align*}
&f_{1}(x,y,z,w)=x^{2}+xy+y^{2}-a=0,\\
&f_{2}(x,y,z,w)=y^{2}+yz+z^{2}-b=0,\\
&f_{3}(x,y,z,w)=z^{2}+zw+w^{2}-c=0,\\
&f_{4}(x,y,z,w)=w^{2}+wx+x^{2}-d=0.
\end{align*}
By computing the resultants $f_5(x,z):=R(f_1,f_2;y)$, $f_6(x,z):=R(f_3,f_4;w)$, and $g(x):=R(f_5,f_6;z)$ (which are actually the similar computations we did in the proof), we will get $g(x)=0$, and $g(x)$ is a quadratic polynomial on $x^2$, which is an analogue of Equation (\ref{eq:quad}). The proof is finished by analyzing the coefficients in $g(x)$, as we did for $k_0,k_1,k_2$ in Equation (\ref{eq:quad}).

 Several interesting questions remain about the function $f_{k}(n,H)$ when $H$ is a longer even cycle. One that immediately arises from Theorem \ref{thm:SomeknownResults} (ii) and Theorem~\ref{thm:f12C4}.
 \begin{problem}\label{problem:Remian1}
   For any integer $\ell\geqslant 2$, estimate the smallest $k$ such that $f_{k}(n,C_{2\ell})=\Theta(n).$
 \end{problem}
  The $\ell=2$ case of Problem \ref{problem:Remian1} is the main topic of this paper. Deriving a similar bound for $f_{2}(n,C_{4})$ is likely to be difficult. The next case, when $\ell=3$, now seems an attractive candidate for further exploration. The idea of using resultants of polynomials mentioned above may be useful, and we suspect our method could be used to obtain some good upper bounds on $k$ for the general $\ell\geqslant 3$.
  We are also interested in the problem of $f_{2}(n,C_{2\ell})$ and we provide the following conjecture.
\begin{conjecture}\label{conj:evencycle}
 For any $\ell\geqslant 3,$ $f_{2}(n,C_{2\ell})=\Omega(n^{2-\frac{2}{\ell}}).$
\end{conjecture}
Conlon and Tyomkyn~\cite{Conlon2020} verified this conjecture when $\ell=3.$ The proof relies on the upper bound for $\textup{ex}(n,\theta_{\ell,t})$~\cite{thetagraph1983} and the observation that the endpoints of theta graph $\theta_{\ell,t}$ cannot be in the same part when $\ell$ is odd (this key observation is also useful in \cite{BukhTailtheta2018}).

\textbf{Note added:} Very recently, Janzer~\cite{Janzer2020} developed a method for finding suitable cycles of given length and then proved Conjecture~\ref{conj:evencycle} in a more general form.

\section*{Acknowledgements}
 Yifan Jing would like to thank J\'{o}zsef Balogh for helpful discussions. The authors express their gratitude to the anonymous reviewer for the detailed and constructive comments which are very helpful for the improvement of the presentation of this paper.

\bibliographystyle{abbrv}
\bibliography{Zt_Jyf_Xzx_Color}

\end{document}